\documentclass{amsart}

\usepackage{latexsym}
\usepackage{amssymb}
\usepackage{graphicx}

\usepackage[hang,small,bf]{caption}

\usepackage{color}
\usepackage{fullpage}

\newtheorem{THM}{Theorem}
\newtheorem{Lemma}[THM]{Lemma}

\newtheorem{Cor}[THM]{Corollary}

\newtheorem{rem}{Remark}

\newcommand{\bR}{\mathbb{R}}
\newcommand{\bC}{\mathbb{C}}
\newcommand{\be}{\begin{equation}}
\newcommand{\ee}{\end{equation}}

\newcommand{\Zed}{\mathbb{Z}}

\newcommand{\lp}{\left(}
\newcommand{\rp}{\right)}
\newcommand{\lb}{\left[}
\newcommand{\rb}{\right]}
\newcommand{\lc}{\left\{}
\newcommand{\rc}{\right\}}
\newcommand{\lab}{\left|}
\newcommand{\rab}{\right|}

\newcommand{\sB}{\mathcal{B}}

\newcommand{\GE}{\Gamma_{\mathrm{even}}}
\newcommand{\GO}{\Gamma_{\mathrm{odd}}}
\newcommand{\TV}{\mathrm{TV}}

\renewcommand{\d}[1]{\,d#1}

\newcommand{\eps}{\varepsilon}

\DeclareMathOperator{\supp}{supp}
\DeclareMathOperator{\set}{set}
\DeclareMathOperator{\card}{card}

\begin{document}

\title{A pseudo-probabilistic approach to the dilation equation for wavelets}
\author[Dumnich \and Neel]{Sarah Dumnich \and Robert W.\ Neel}
\address{Department of Mathematics, Frostburg State University, Frostburg, MD, USA}
\email{scdumnich@frostburg.edu}
\address{Department of Mathematics, Lehigh University, Bethlehem, PA, USA}
\email{robert.neel@lehigh.edu}

\begin{abstract} The dilation equation arises naturally when using a multiresolution analysis to construct a wavelet basis. We consider solutions in the space of signed measures, which, after normalization, can be viewed as pseudo-probability measures. Using probabilistic ideas as well as a notion of binary expansions, we discuss the existence and absolute continuity of solutions in dimensions one and two, efficiently recovering previous results in two natural cases. A central role is played by the development of two variants of the classical cascade algorithm, which are adapted to the pseudo-probabilistic and elementary number theoretic context.
\end{abstract}

\maketitle

\section{Introduction}

In the 80's, Mallat \cite{Mallat89} and Meyer \cite{Meyer} formalized multiresolution analysis (MRA), which set the groundwork for the construction of wavelet bases. A central role in this approach is played by the dilation equation; indeed, a solution to the dilation equation, called a scaling function, canonically determines a corresponding wavelet basis \cite{Daubechies88}. Let $\Gamma \subset \mathbb{R}^d$ be a lattice and let $M$ be an integer-valued \textit{expanding} matrix. That is, all eigenvalues of $M$ are greater than 1 in absolute value, so that $M$ preserves the lattice. Then a \textit{dilation equation} is an equation of the form 
\begin{equation}\label{dile}
\phi(x)= |\det M | \sum_{k \in \Gamma} p_{k}\phi(Mx- k),
\end{equation}
where $\{p_k\}$ is a sequence (indexed by the countably infinite set $\Gamma$) of real numbers. If $\{p_{k}\} $ is in $l^2(\Gamma)$, then the dilation equation always has a solution $\phi$ in the distributional sense (see \cite{Daubechies92}). Functional solutions (meaning solutions $\phi$ that are functions) to dilation equations are useful in many applications such as subdivision schemes, interpolation methods, and the construction of wavelet bases of $L^2(\mathbb{R}^d)$ \cite{Strang, Daubechies92, Heil94}.  Depending on conditions placed upon $\{p_{k}\}$, a functional solution can be a scaling function (meaning that its translates form an orthonormal basis in the MRA) or merely a prescale function (meaning that its translates form a Riesz basis).

Dobric, Gundy, and Hitczenko (in the one-dimensional case) \cite{DobricGundy} and Belock and Dobric (in the multi-dimensional case) \cite{Dobric} considered a probabilistic approach to equation \eqref{dile}. This is natural because, under the condition that all $p_{\gamma}$ are non-negative, the right-hand side of the dilation equation can be interpreted as the convolution of two probability measures. This allowed them to make use of probabilistic tools and to look for (generalized) solutions given by probability measures, which form a natural class intermediate between functional and distributional solutions. They then further discuss conditions under which the probability measure was absolutely continuous and thus had a density which gives a functional solution. Though this approach is appealing, it suffers from the serious drawback that one does not wish to restrict attention to cases in which both the $p_{\gamma}$ and the $\phi$ are single-sign. Most of the important and well-known scaling functions, including those of Daubechies' wavelets \cite{Daubechies88}, change sign. Thus, the starting point of the present paper is to generalize this approach to allow (finite) signed measures as solutions to the dilation equation. From one perspective, this means (possibly up to normalization) replacing probability measures by ``pseudo-probability measures'' (see Baez-Duarte \cite{Duarte} for an introduction and further references) This introduces some technical complications, since, for example, the Portmanteau Theorem and L\'{e}vy Convergence Theorem do not carry over to this generalized setting.

A second line of research which informs the present paper is the use, by Gundy-Jonsson \cite{Gundy10} and Gundy \cite{GundySurvey}, of a binary sequence encoding of elements of $\bR$ and $\bR^2$ in order to connect scaling functions on these two spaces. We make consistent use of this encoding, and for this reason we restrict our attention to appropriate dilations in dimensions one and two.

Measure-valued solutions to the dilation equation \eqref{dile} are given by solutions to the dilation equation for measures \eqref{Eqn:DEForMeas}. We consider the most common choices for $M$ in dimensions one and two, and we consider the cases when $\{p_k\}_{k\in\Gamma}$ has finite support and satisfies either a probability condition or a type of orthonormality condition. In terms of the existence and uniqueness of a measure-valued solution and the existence of a density, we give an efficient and unified proof of the key results of Belock and Dobric under the probability condition and of the classical results (as given in \cite{Lawton97}) under the orthonormality condition. In contrast to most of the wavelet literature, we make no use of the Fourier transform, instead choosing to highlight the ideas coming from pseudo-probability and binary expansions. Moreover, we develop two recursive schemes for computing the solution measure $\mu$, both of which are probabilistically-motivated variants on the classical cascade algorithm (see \cite{Lawton97} for a description in any dimension). One gives $\mu$ as the weak limit of a sequence of discrete pseudo-probability measures; these approximating discrete measures arise naturally from the binary sequence representation mentioned above. The second scheme proceeds by computing the conditional expectation of the pseudo-probability on successively more refined $\sigma$-algebras. Since these $\sigma$-algebras are generated by translations and dilations of tiles, these conditional pseudo-probabilities are equivalent (up to linear algebra) to the wavelet coefficients of the density of $\mu$ with respect to the Haar basis of the tiles. Thus these approximations are natural from the perspective both of probability and of harmonic analysis and wavelet bases.

We note that probabilistic methods have been used to good effect when $p_k$ is allowed to have infinite support in order to give necessary conditions and sufficient conditions for a (functional) solution to the dilation equation to be a scaling function in \cite{DobricGundy}, \cite{Gundy2000}, \cite{Gundy07}, \cite{Gundy10}, and \cite{GundySurvey} (indeed, the binary coding mentioned above was developed in this context). This line of research was extended to arbitrary dimension and to pre-scale functions by Curry \cite{Curry}. We also note that an alternative approach to multi-dimensional wavelets, in which additional dilations are introduced to simplify the geometry, is described in \cite{LWW}.
 
The structure of the paper is as follows. In Section \ref{Sect:Prelim}, we give the relevant background on our choices of dilation matrix $M$, the elementary number theory behind the corresponding binary encoding scheme, the dilation equation for measures, and conditional pseudo-probability. In Section \ref{Sect:MainResults}, we introduce our series of discrete approximating measures and prove their weak convergence, giving the existence of a the unique solution measure $\mu$, under both the probability and orthonormality conditions. We also give results on the existence of a density for $\mu$; this is automatic under the orthonormality conditions, but requires additional assumptions on the $p_k$ under the probability conditions. In Section \ref{Sect:OneDComp}, we briefly present the recursive computation of $\mu$ via conditional expectations in the one-dimensional case by illustrating it for the D4 wavelet. In Section \ref{Sect:GundyMethod}, we return to the binary encoding scheme and make the point that, while it gives a useful correspondence between scaling functions in one and two dimensions, this correspondence does not preserve continuity of such functions, and thus does not completely reduce the study of dimension two to that of dimension one. In Section \ref{Sect:TwoDComp}, we develop the computation of $\mu$ by conditional expectations in more detail in the two-dimensional case, where the geometry is more complicated. We do this by giving a pair of examples illustrating the approach, one each under the probability and orthonormality conditions.

\section*{Acknowledgements} This paper is partially based on the first author's PhD dissertation. The problem was originally suggested by her advisor, Vladimir Dobric. Unfortunately, Dobric passed away before the dissertation could be written, and the second author took over his role as advisor.

We are grateful to Richard Gundy for helpful discussions.

\section{Preliminaries}\label{Sect:Prelim}

We briefly discuss the elementary number theory underlying our approach.

Given a dilation $M$ on $\bR^n$, let $\mathcal{D}$ be a complete set of coset representatives for $\mathbb{Z}^n/M(\mathbb{Z}^n)$. We assume that $\mathcal{D}$, called the \textit{digit set}, contains the zero vector. Let $\mathbb{P}$ denote the set of all $k \in \mathbb{Z}^n$ that can be written as a finite sum of the form $$k = \sum_{j=0}^{N(k)}M^jd_j$$ with $d_j \in \mathcal{D}$. The pair $(M, \mathcal{D})$ is called a number system if $\mathbb{P} = \mathbb{Z}^n$. In this case $M$ is said to be the \textit{radix} of the system. If the digit set consists of all nonnegative multiples, $m=0, 1, ..., (q-1)$, of a single coordinate unit vector the system is called canonical \cite{Gundy10}.

Lagarias and Wang \cite{Lagarias} have classified all expanding matrices in $\mathbb{R}^2$, up to integral similarity by a unimodular matrix $U \in M_2(\mathbb{Z})$. Their list is as follows: if $\det (M)=-2$, 
\[
M \sim C_1 = \left( \begin{array}{cc}
0 & 2 \\
1 & 0 \end{array} \right)
\]
is the canonical representative of the class. If $\det (M) =2$ there are five classes, defined by the following canonical representatives:
\[
C_2 = \left( \begin{array}{cc}
0 & 2 \\
-1 & 0 \end{array} \right)\\
\]
\[
\pm C_3 = \pm \left( \begin{array}{cc}
1 & 1 \\
-1 & 1 \end{array} \right)\\
\]
\[
\pm C_4 = \pm \left( \begin{array}{cc}
0 & 2 \\
-1 & 1 \end{array} \right).
\]
For each of these cases, a digit set $\mathcal{D}$ exists such that the set $T(M, \mathcal{D}) = \left\{ \sum_{j=1}^{\infty} M^{-j} d_j \right\}$ is a \textit{tile} \cite{Gundy10}, where a tile $T$ is a subset of the plane such that translations of $T$ by Gaussian integers $\gamma$ are disjoint up to a set of Lebesgue measure 0 and $\cup_{\gamma} \lp T+\gamma \rp = \mathbb{R}^2$. The following theorem from Gundy and Jonsson \cite{Gundy10} summarizes their results regarding these classes of dilation.
\begin{THM}
For no choice of $\mathcal{D}$ is $(C_1, \mathcal{D})$ a number system. The matrices $C_2$, $-C_3$, $\pm C_4$ all generate number systems with the canonical digit set $\mathcal{D}_1=\{0, \epsilon_1\}$, where $\epsilon_1 = (1, 0)'$. The pair $(+C_3, \mathcal{D}_1)$ generates a self-affine tile $T(+C_3, \mathcal{D}_1)$, but for no digit set $\mathcal{D}$ is $(+C_3, \mathcal{D})$ a number system.
\end{THM}

Thus, the pair $(+C_3, \mathcal{D}_1)$ is the exceptional case in the list in that it generates a self-affine tile but does not generate a number system, and this is the case we focus on. We find it easier to identify $\mathbb{R}^2$ with the complex plane $\mathbb{C}$ in order to simplify computations and notation. In this case, multiplication by the matrix $+C_3$ is equivalent to multiplication by $1+i$. The expressively-named \textit{Twin Dragon} (see Figure \ref{Fig:TwinD}) is the tile which is generated by $(+C_3, \mathcal{D}_1)$ and can we written as the following: 
\begin{equation}\label{Eqn:TwinDragon}
T= \left\{   \sum_{j=1}^{\infty} \frac{\gamma_j}{(1+i)^j} :  \gamma_j \in \{ 0, 1 \} \quad \forall j  \right\}.
\end{equation}

\begin{figure}
\centering
  \includegraphics[width=\linewidth]{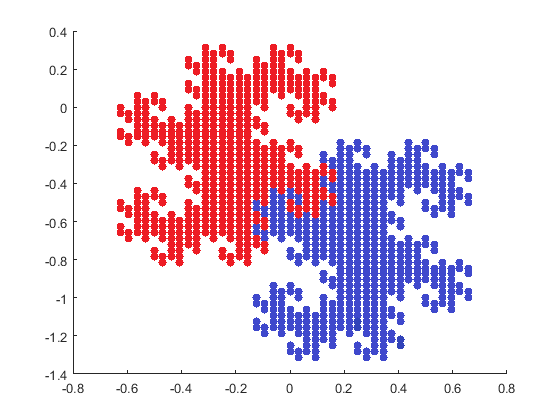}
  \caption{The Twin Dragon $T$. The two sub-regions are the half-tiles, corresponding to points having ``binary expansion'' beginning with 0 and 1, respectively.}
\label{Fig:TwinD}
\end{figure}

Further, we see that there is a natural parallelism with the real line. Of course, the unit interval $[0,1]$ tiles $\bR$ (under translations by the integers), and the unit interval can be expressed in binary as
\[
[0,1]= \left\{   \sum_{j=1}^{\infty} \frac{\gamma_j}{2^j} : \gamma_j \in \{ 0, 1 \} \quad \forall j  \right\},
\]
which is analogous to the construction of the twin dragon. In particular, every real number can be expressed an an integer plus a binary expansion of the above form. Similarly, every point in $\bR^2$ can be expressed as a Gaussian integer plus a binary expansion in the sense of equation \eqref{Eqn:TwinDragon}.

We wish to consider the dilation equation on $\bR$ when the lattice is $\Zed$ and $M$ corresponds to multiplication by $2$ and on $\bR^2\cong \bC$ when the lattice is $\Zed^2$ and $M$ corresponds to multiplication by $1+i$, and we wish to so so in a unified way. Motivated by the above, in what follows we let $\Gamma$ denote either $\Zed$ or $\Zed^2$ and $M$ multiplication by $2$ or $1+i$, as appropriate. Of course, this is consistent with the general dilation equation for functions \eqref{dile}, but moreover, it means that for either $\bR$ or $\bR^2$, we can represent points in the form $k+\sum_{j=1}^{\infty} M^{-j}\gamma_j$ where $k\in\Gamma$ and $\gamma_j\in\{0,1\}$. Note that reinterpreting $M$ as either $2$ or $1+i$ gives a map from $[0,1]$ to $T$ which is well-defined up to a set of measure zero (namely, the set of points that have non-unique ``binary'' expansions). In fact, this map is an $L^2$ isomorphism (where $L^2$ is understood with respect to Lebesgue measure in the relevant dimension). Further, any bijection between $\Zed$ and $\Zed^2$ extends this to an isometry between $L^2(\bR)$ and $L^2(\bR^2)$. This idea underlies Gundy's work \cite{Gundy10} connecting scaling functions on $\bR$ and $\bR^2$, which we discuss in a bit more detail in Section \ref{Sect:GundyMethod}.

In this paper, we deal primarily with the dilation equation for measures, namely
\begin{equation}\label{Eqn:DEForMeas}
\mu (A) = \sum_{k\in \Gamma} p_k \mu (MA-k) \quad\text{for all Borel sets $A$},
\end{equation}
where $\mu$ is a pseudo-probability measure, that is, a (finite) signed measure such that $\mu(\bR^n)=1$. As mentioned, we develop two recursive methods for computing a solution $\mu$, both of which are related to binary expansions. First, in Section \ref{Sect:DiscApprox}, we develop a sequence of discrete approximations, where the $n$th approximation is supported on points which have binary expansions terminating after $n$ places (thus essentially working on the dyadic rationals and the natural two-dimensional analogue). Second, in Sections \ref{Sect:OneDComp} and \ref{Sect:TwoDComp}, we compute $\mu$ via (a version of) the conditional expectation of $\mu$ with respect to the underlying Lebesgue measure, and which we now explain.

First, assume that $\mu$ is absolutely continuous (with respect to Lebesgue measure), and let $T'$ denote either $[0,1]$ or the twin dragon $T$, as appropriate (that is, $T'$ is the fundamental tile). Now let $\Sigma_0$ be the $\sigma$-algebra generated by all such translates of $T'$, that is, $\Sigma_0$ is generated by $\{k+T'\}_{k\in\Gamma}$. Then let $\Sigma_n=M^{-1}\Sigma_{n-1}=M^{-n}\Sigma_0$ for $n=1,2,\ldots$. Note that the $\Sigma_n$ are nested, $\Sigma_n\subset \Sigma_{n+1}$, and the collection $\{\Sigma_n\}_{n\in\{0,1,2,\ldots\}}$ generates the Borel $\sigma$-algebra. (The nesting property shows that the $\Sigma_n$ are in fact a filtration.) Then the restriction $\mu|_{\Sigma_n}$ of the pseudo-probability measure $\mu$ to $\Sigma_n$ is the conditional expectation of the density of $\mu$ given $\Sigma_n$. Further, knowing $\mu|_{\Sigma_n}$ for all $n$ determines $\mu$ (since it determines the density), and we will develop a method for recursively computing them in this case. In particular, we can recursively compute the measures of the (dilated, translated) tiles, $\mu\lp M^{-n}(k+T')\rp$, and since $\mu$ is absolutely continuous, this determines $\mu|_{\Sigma_n}$. As mentioned, this is natural from a measure-theoretic perspective, but also from a wavelet point of view. In particular, the Haar wavelets (where, in $\bR^2$, we understand the Haar wavelets with respect to the twin dragon) at scale $n$ are $\Sigma_n$-measurable. Thus, if $\mu$ has a density $\phi$ in $L^2$, knowing $\mu|_{\Sigma_n}$ is equivalent, up to linear algebra, to knowing the wavelet coefficients of $\phi$ with respect to the scale-$n$ Haar wavelets.

In the case when $\mu$ is not absolutely continuous, we nonetheless think of $\mu|_{\Sigma_n}$ as a natural generalization of the conditional expectation. However, in this case, $\mu|_{\Sigma_n}$ is not necessarily determined by knowing $\mu\lp M^{-n}(k+T')\rp$ for all $n$ and $k$, essentially because the tiles overlap on their boundaries (and thus the tiles don't form a $\pi$-system, and the boundaries need not be $\mu$-null). The natural solution is to modify the tiles in order to give a true partition at every scale, not merely a tiling. In dimension 1, this is straightforward. Let $\tilde{T}=(0,1]$, and let $\tilde{\Sigma}_n$ be the $\sigma$-algebra generated by $\{M^{-n}(k+\tilde{T})\}_{k\in \Zed}$, analogous to the above. Then the $\tilde{\Sigma}_n$ are nested, giving a filtration (which means we don't just have a partition for each $n$, but one which is compatible with the dilation equation in the obvious way), and further, the measures $\mu\lp M^{-n}(k+\tilde{T})\rp$ of the translates of $M^{-n}\tilde{T}$ determine $\mu|_{\tilde{\Sigma}_n}$. Thus, as before, we can recursively compute the (generalized) condition expectations $\mu|_{\tilde{\Sigma}_n}$, which determine $\mu$, now whether or not $\mu$ is absolutely continuous.

Extending this procedure to two dimensions requires a modification of the twin dragon (by removing part of the boundary)  to give a partition of $\bR^2$ compatible with the dilation equation (that is, to get a two-dimensional version of $\tilde{T}=(0,1]$). Unfortunately, such a modification of the twin dragon does not appear in the literature, as far as we know, and it seems too far afield to pursue the construction of such a partition in the present paper. (We briefly return to this point in Section \ref{Sect:TwoDComp}.)

\section{Solving the dilation equation for measures}\label{Sect:MainResults}

As already discussed, we wish to consider the dilation equation equation for (finite, signed) measures \eqref{Eqn:DEForMeas} on $\bR$ when the lattice $\Gamma$ is $\Zed$ and $M$ corresponds to multiplication by $2$ and on $\bR^2\cong \bC$ when the lattice $\Gamma$ is $\Zed^2$ and $M$ corresponds to multiplication by $1+i$. Again, we assume $\{p_k\}_{k\in\Gamma}$ is a real-valued function on $\Gamma$ with finite support. The consideration of the dilation equation for measures is motivated in part by the fact that if a measure satisfies this equation and it has a density, then the density satisfies the dilation equation for functions almost everywhere. In particular, if the density is continuous, it satisfies the dilation equation for functions without qualification.

\begin{Lemma}
The density $\phi$ of an absolutely continuous solution $\mu$ of the dilation equation for signed measures satisfies the functional dilation equation almost everywhere.
\end{Lemma}

\begin{proof}
Let $\phi$ be the density of the solution to the dilation equation for signed measures, $\mu$.  Let $x_0$ be any point in $\supp ( \mu )$ and $r \in \mathbb{R}^+$.  Let $B(x_0, r)$ denote the ball of radius $r$ about the point $x$. Then applying the dilation equation for measures to $B(x_0, r)$ and using that $\mu$ has a density gives
\[
\begin{split}
 \int_{B(x_0, r)} \phi(x) dx &=\sum_k p_k \int_{M(B(x_0, r))-a_k}\phi(y) dy\\
\quad \quad &= \sum_k p_k |\det M| \cdot \int_{B(x_0, r)}\phi(Mx-a_k) dx . \\
\end{split}
\]
Now, taking the limit as $r \rightarrow 0$, the Lebesgue Differentiation Theorem gives the conclusion.
\end{proof}

Further, we suppose the coefficients $\{p_k\}$ satisfy the normalization condition
\[
\sum_{k\in\Gamma} p_k = 1 .
\]
Let $\mu_1 =\sum_{k\in\Gamma}p_{k} \delta_{M^{-1}k}$ (here $\delta_x$ is the delta mass at $x$), and note that the normalization condition implies that $\mu_1$ is a pseudo-probability measure. Let $D(x)=M^{-1}x$ (for $x$ in $\bR$ or $\bR^2$, as appropriate), and let $D_{\star}$ be the corresponding push-forward map on signed measures. Then we see that the dilation equation for measures can be re-written as
\[
\mu=\mu_1 \star D_{\star}(\mu).
\]

\subsection{Discrete approximate solutions}\label{Sect:DiscApprox}
The above suggests solving the dilation equation by a natural sequence of approximate solutions; namely, for $n\geq 2$, we iteratively define the pseudo-probability measures $\mu_n$ by $\mu_n=\mu_1 \star D_{\star}(\mu_{n-1}).$
Equivalently, we have 
\begin{equation}\label{Eqn:MuN}
\mu_n=\bigstar_{k=0}^{n-1} (D_{\star})^{k}(\mu_1) .
\end{equation}
Recall that weak convergence of signed measures is defined in terms of their integrals against bounded, continuous test functions. First, we observe that, if these approximations converge, the limit gives a solution to the dilation equation.

\begin{Lemma}
If $ \mu_n $ converges weakly to a limiting measure $\mu$, then $\mu$ satisfies the dilation equation for measures.
\end{Lemma}

\begin{proof}
Assume that the limit exists and $\mu = \lim_{n \rightarrow \infty} \mu_n$. Then, using that both convolution and $D_{\star}$ are continuous in the weak topology, we have 
\[\begin{split}
\mu =& \lim_{n \rightarrow \infty} \bigstar_{k=0}^{n} (D_{\star})^{k}(\mu_1)\\
		=& (\mu_1) \star  \lim_{n \rightarrow \infty} \bigstar_{k=1}^{n} (D_{\star})^{k}(\mu_1)\\
		=& (\mu_1) \star  \lim_{n \rightarrow \infty} (D_{\star})\lp \bigstar_{k=0}^{n-1} (D_{\star})^{k}(\mu_1) \rp\\
		=& (\mu_1) \star (D_{\star})\lp \lim_{n \rightarrow \infty}  \bigstar_{k=0}^{n-1} (D_{\star})^{k}(\mu_1) \rp\\
		=& (\mu_1) \star (D_{\star})(\mu),
\end{split}\]
so that $\mu$ satisfies the dilation equation.
\end{proof}

\subsection{The support of $\mu_n$ and the radix expansion}\label{Sect:Support}

By construction, $\mu_n$ is supported on $M^{-n} \Gamma$. In what follows, it will be convenient to define the functions $w_n(x)=\mu_n\lp \{x\}\rp$, so that $\mu_n = \sum_{x\in M^{-n}\Gamma} w_n(x)\delta_x$.

Let $\tilde{\Gamma}\subset\Gamma$ be the support of $\{p_k\}$ (which is finite by assumption). Then if 
\begin{equation}\label{Eqn:SuppOfMuN}
S_n=\lc \sum_{j=1}^{n} M^{-j}\tilde{\gamma}_j : \tilde{\gamma}_j\in\tilde{\Gamma} \text{ for all $j$}\rc ,
\end{equation}
from equation \eqref{Eqn:MuN}, we see that $\supp\lp \mu_n\rp \subset S_n$, so that the support of $\mu_n$ is finite. Further, let $|M|$ be $2$ if $M=2$ and $\sqrt{2}$ if $M=1+i$ (that is, $|M|$ is the Euclidean norm of $M$ as a complex number). Then for any $n$ and any $x\in S_n$, there is some sequence of $\tilde{\gamma}_j$ in $\tilde{\Gamma}$ such that
\[
 |x|  = \left|\sum_{j=1}^{n} M^{-j} \tilde{\gamma}_j\right| \leq \max_{\tilde{\gamma}\in\tilde{\Gamma}} |\tilde{\gamma}| \sum_{j=1}^{\infty} |M|^{-j} =  \frac{\max_{\tilde{\gamma}\in\tilde{\Gamma}} |\tilde{\gamma}|}{|M|-1} .
\]
Hence, the $S_n$ are uniformly bounded (as sets in Euclidean space). In particular, if we let $R$ be $\frac{\max_{\tilde{\gamma}\in\tilde{\Gamma}} |\tilde{\gamma}|}{|M|-1}$, then $R$ depends only on $M$ and $\tilde{\Gamma}$, and the ball of radius $R$ around the origin, which we denote by $B_R(0)$, contains the supports of all the $\mu_n$ and thus the support of their limit, assuming such a limit exists. Also, in terms of cardinality, by counting (dilated) lattice points in a fixed ball, we see that there is some $N$ (again, depending on $M$ and $\tilde{\Gamma}$) such that
\begin{equation}\label{Eqn:CardOfSupport}
\card(\supp(\mu_n)) \leq \card(S_n)\leq \card\lp B_R(0) \cap M^{-n}\Gamma \rp \leq N \cdot 2^n.
\end{equation}

Next, for $x\in S_n$, let $\Lambda_n(x)$ be the set of $n$-tuples $(\tilde{\gamma}_1,\ldots, \tilde{\gamma}_n)$ of points in $\tilde{\Gamma}$ such that $x= \sum_{j=1}^{n} M^{-j}\tilde{\gamma}_j $; note that $\Lambda_n(x)$ will be finite. Then we see that
\[
\mu_n\lp \{x\}\rp = w_n(x) =  \sum_{\substack{(\tilde{\gamma}_1,\ldots, \tilde{\gamma}_n)\in\Lambda_n(x)}} \lp\prod_{j=1}^n p_{\tilde{\gamma}_j}\rp .
\]
In particular, we will have that $\supp(\mu_n)=S_n$ except in the case when the right-hand side of the above vanishes for some $x\in S_n$.

This gives an appealing interpretation of the $\mu_n$. Namely, we know that when $\tilde{\Gamma}=\{0,1\}$ is the standard set of binary digits, we get every point of the lattice $M^{-n} \Gamma$ in the standard tile (either $[0,1]$ or the twin dragon $T$) uniquely as
\[
\lc \sum_{j=1}^{n} M^{-j}\gamma_j : \gamma_j\in\{0,1\} \text{ for all $j$}\rc .
\]
If $\tilde{\Gamma}=\{0,1\}$ and $p_0=p_1=1/2$, then in the limit $\mu_n$ converges weakly to Lebesgue measure on the standard tile; indeed, this is the standard construction of the uniform (with respect to Lebesgue measure) probability measure on the standard tile via binary expansion. Note that even in the limit as $n\rightarrow \infty$, all points, except for a set of (Lebesgue) measure zero, in the standard tile have a unique binary representation. In the general case, $\tilde{\Gamma}$ is a non-standard ``too large'' set of digits, which, when used in the binary expansion \eqref{Eqn:SuppOfMuN}, gives lattice points not necessarily confined to the standard unit tile and, in general, gives them with non-unique representations. Then the pseudo-probability of such a point is the sum over all such representations of the products of the associated pseudo-probabilities of each digit. Intuitively, we have a pseudo-probability generalization of the standard construction of a uniform probability measure by binary expansion.

\subsection{An existence condition and uniqueness}

\begin{THM}
Suppose $\|\mu_n\|_{\mathrm{TV}}$ is bounded independently of $n$. Then $\mu_n$ converges weakly to a pseudo-probability measure $\mu$.
\end{THM}
\begin{proof}
Let $f$ be a bounded, continuous function (on either $\bR$ or $\bR^2$); it suffices to show that $\int f\d\mu_n$ converges. In particular, we will show that $\int f\d\mu_n$ gives a Cauchy sequence. Because all the $\mu_n$ are supported on a fixed compact set, namely $\overline{B_R(0)}$, we may assume that $f$ is uniformly continuous. Let $A>0$ be such that $\|\mu_n\|_{\mathrm{TV}} <A$ for all $n$.

Choose $\eps>0$. By uniform continuity, there exists $\delta>0$ such that $|f(x)-f(y)|<\frac{\eps}{A^2}$ whenever $\|x-y\|<\delta$. Now choose $N$ such that $D_{\star}^n \lp \overline{B_R(0)}\rp\subset B_{\delta}(0)$ whenever $n\geq N$.

Choose $n>N$ and $k>0$. From \eqref{Eqn:MuN}, we have that $\mu_{n+k}=\mu_n \star D_{\star}^n \lp \mu_k\rp$. Then we can write
\[\begin{split}
\lab \int f\, d\mu_n -\int f\, d\mu_{n+k}\rab &= 
\lab \sum_{x\in\supp(\mu_n)} \lp w_n(x)f(x)- w_n(x)\sum_{y\in \supp(\mu_k)}w_k(y) f\lp x+ M^{-n}y\rp\rp\rab \\
& = \lab \sum_{x\in\supp(\mu_n)} w_n(x)\sum_{y\in \supp(\mu_k)}w_k(y)\lp f(x)- f\lp x+ M^{-n}y\rp\rp\rab ,
\end{split}\]
where, in the second line, we used that $\sum_{y\in \supp(\mu_k)}w_k(y) =1$, which follows from $\mu_k$ being a pseudo-probability measure. From here, using that $n>N$, we have
\[\begin{split}
\lab \int f\, d\mu_n -\int f\, d\mu_{n+k}\rab &\leq
\lp\sum_{x\in\supp(\mu_n)} \lab w_n(x)\rab \rp\lp\sum_{y\in \supp(\mu_k)}\lab w_k(y)\rab\rp \\
& \qquad\times \lp\sup_{\substack{x\in\supp(\mu_n)\\ y\in \supp(\mu_k)}} \lab f(x) - f\lp x+ M^{-n}y\rp\rab\rp \\
& \leq \|\mu_n\|_{\mathrm{TV}} \|\mu_k\|_{\mathrm{TV}} 
\sup_{\substack{x\in\supp(\mu_n)\\ z\in B_{\delta}(0)}} \lab f(x) - f\lp x+ z \rp\rab \\
& \leq A\cdot A \cdot \frac{\eps}{A^2} \\
&=\eps.
\end{split}\]
Since this holds whenever $n>N$ and $k>0$, we see that $\int f\,d\mu_n$ gives a Cauchy sequence, completing the proof.
\end{proof}

In particular, we see that if the total variation norm of the $\|\mu_n\|$ is uniformly bounded, then a solution to the dilation equation for measures exists. The corresponding uniqueness result is straightforward. First note that if $\mu$ solves the dilation equation, so does $c\mu$ for any $c\in\bR$. However, aside from this scaling, the solution is unique. Of course, this follows from the general fact that the solution is unique in the space of distributions, but for completeness, we quickly give the measure-theoretic proof.

\begin{THM}
Suppose the $\mu_n$ converge. Then $\mu = \lim _{n \rightarrow \infty} \mu_n$ is the unique solution, up to scaling by a constant, for the signed measure dilation equation.
\end{THM}

\begin{proof}
Suppose $\nu$ is any (finite) signed measure; recall that by definition $\nu$ has bounded total variation.  Then it is easy to see that $D_{\star}^{n}\nu \rightarrow \nu(\bR^m)\delta_0$ weakly as $n \rightarrow \infty$ (here $m$ is either 1 or 2, as appropriate). Now, let $\tilde{\mu}$ be any solution to the dilation equation for measures. Then $\tilde{\mu}=\mu_1 \star D_{\star}(\tilde{\mu})$, and iterating this gives
\[
\tilde{\mu} =\lp \mu_0\star D_{\star}\mu_0 \star D_{\star}^{2}\mu_0  \dots \star D_{\star}^{n}\mu_0 \rp \star D_{\star}^{n+1}\tilde{\mu} .
\]
Letting $n\rightarrow \infty$ on the right, we use the definition of $\mu$ and the above to see that
\[
\tilde{\mu} = \mu\star  \tilde{\mu}(\bR^m)\delta_0 = \tilde{\mu}(\bR^m) \mu .
\]
Thus, $\mu$ and $\tilde{\mu}$ differ by a constant, which is what we wanted to show.
\end{proof}

At this point, we have shown that, under any additional assumptions on the coefficients $p_k$ that give that $\|\mu_n\|_{\mathrm{TV}}$ is uniformly bounded, the dilation equation for measures has a unique solution, and this solution is given as the weak limit of the $\mu_n$. We now consider two situations in which we have such a bound.

\subsection{Probability measures}\label{Sect:ProbConds}

If $p_k\geq 0$ for all $k$, then $\mu_1$ will be a probability measure (rather than just a pseudo-probability measure). This is the case considered by Belock and Dobric \cite{Dobric}. In this case, every $\mu_n$ will be a probability measure, so trivially we have $\|\mu_n\|_{\TV}=1$ for all $n$, giving a solution measure $\mu$ which is also a probability measure. However, $\mu$ will not have a density, in general.

With additional assumptions on the $p_i$, though, $\mu$ will be absolutely continuous. Note that the action of $M$ splits $\Gamma$ into two cosets. We call $M\Gamma=\GE$ the evens, because on $\bR$ when $\Gamma=\Zed$, it is exactly the even integers, and we call the other coset $\GO$ the odds, for the analogous reason. Then the condition for $\mu$ to be absolutely continuous is that $\sum_{k\in\GE}p_k=\sum_{k\in\GO}p_k=1/2$ (as originally proved in \cite{Dobric}). The significance of this condition is explained by the following result.

\begin{Lemma}\label{Lem:WProb}
Assume that $p_k\geq 0$ for all $k\in\Gamma$ and that $\sum_{k\in\GE}p_k=\sum_{k\in\GO}p_k=1/2$. Then for all $n$,
\[
0\leq w_n(x)\leq \frac{1}{2^n} \quad \text{for all $x\in M^{-n}\Gamma$.}
\]
\end{Lemma}
\begin{proof}
That the $\mu_n$ are probability measures implies the non-negativity of the $w_n$.
For the upper bound, we proceed by induction. When $n=1$, this is just the observation that $p_k\leq 1/2$ for all $k$, which is immediate from the assumptions on the $p_k$.

Now assume that the conclusion holds for $w_{n-1}$. We know that $\mu_{n-1}$ is supported on $M^{-(n-1)} \Gamma$ and that $\mu_n= (D_{\star})^{n-1}(\mu_1)\star \mu_{n-1}$. Thus, for any $x\in M^{-n}\Gamma$ (where we recall that $\supp(\mu_n)\subset M^{-n}\Gamma$), we have
\[
w_n(x) = \sum_{k\in\Gamma} p_k w_{n-1}\lp x-M^{-n}k\rp .
\]
where we extend $w_{n-1}$ to $M^{-n}\Gamma$ by letting it be zero off of $M^{-n+1}\Gamma$. Now $w_{n-1}\left( x-M^{-n}k\right) $ can only be non-zero if $x-M_{-n}k \in M^{-(n-1)}\Gamma$, and thus there are two cases depending on whether $x \in M^{-(n-1)}\Gamma$ or $x \in M^{-n}\Gamma \setminus M^{-(n-1)}\Gamma$. In the first case, only even $k$ contribute, and in the second, only odd $k$. Then we have
\[\begin{split}
w_n(x) &= \begin{cases} \sum_{k\in\GE} p_k w_{n-1}\lp x-M^{-n}k\rp & \text{if $x \in M^{-(n-1)}\Gamma$} \\
\sum_{k\in\GO} p_k w_{n-1}\lp x-M^{-n}k\rp & \text{if $x \in M^{-n}\Gamma \setminus M^{-(n-1)}\Gamma$} \end{cases} \\
&\leq \begin{cases} \frac{1}{2^{n-1}}\sum_{k\in\GE} p_k  & \text{if $x \in M^{-(n-1)}\Gamma$} \\
\frac{1}{2^{n-1}}\sum_{k\in\GO} p_k & \text{if $x \in M^{-n}\Gamma \setminus M^{-(n-1)}\Gamma$} \end{cases} \\
&= \frac{1}{2^{n-1}}\cdot \frac{1}{2}  \quad \text{for all $x\in M^{-n}\Gamma$,}
\end{split}\]
where we've used the inductive assumption and the conditions on the $p_k$. This completes the proof.
\end{proof}

From here, it's straightforward to show that $\mu$ will have a density under these conditions.

\begin{THM}
Assume that $p_k\geq 0$ for all $k\in\Gamma$ and that $\sum_{k\in\GE}p_k=\sum_{k\in\GO}p_k=1/2$. Then $\mu$ has a density $\phi$ (with respect to Lebesgue measure), and $\phi \in L^{\infty}$ with $||\phi||_{\infty} \leq 1$.
\end{THM}

\begin{proof}
We begin by letting $f$ be a continuous smooth function with compact support. Then, by definition of $\mu$ and $\mu_n$, we have 
\begin{equation*}
\int f d \mu =\lim _{n \rightarrow \infty} \int f d \mu _n
 =\lim _{n \rightarrow \infty} \sum_{x \in S_n} f(x) w_n(x).
\end{equation*}
Using Lemma \ref{Lem:WProb} gives
\[
\int f d \mu  \leq \lim _{n \rightarrow \infty} \frac{1}{2^n} \sum_{x \in S_n} |f(x)| .
\]
Note that $2^{-n}\sum_{k\in\Gamma} \delta_{M^{-n}k}$ converges weakly to Lebesgue measure as $n\rightarrow\infty$. So we have
\begin{equation*}
\int f d \mu \leq \int |f(x)| dx
= ||f||_1.
\end{equation*}
Therefore integration of a smooth function $f$ against $\mu$ is bounded by $||f||_1$.  By the density of such functions in $L^1$, we see that integration against $\mu$ determines a bounded linear functional on $L^1$. By duality, we conclude that $\mu=\phi\, dx$ with $\phi \in L^{\infty}$ and $||\phi||_{\infty} \leq 1$.
\end{proof}

\subsection{Orthogonality conditions and scaling functions}\label{Sect:OrthoConds}

For the remainder of this section, we assume (instead of the probability conditions above) that the $p_k$ satisfy the orthonormality conditions
\begin{equation}\label{Eqn:OrthoConds}
\sum_{k\in\Gamma} p_k p_{k+Mi} = \frac{1}{2}\delta_{0i} \quad\text{for all $i\in\Gamma$} .
\end{equation}
(In the wavelet literature, a coefficient sequence satisfying these conditions is called a conjugate quadrature filter.) It is known that these conditions are necessary for a function satisfying equation \eqref{dile} to generate an orthonormal basis for the associated MRA. They are also known to be almost sufficient (we refer to \cite{Lawton97} for the case when $p_k$ has finite support). Thus, it is especially interesting to consider the dilation equation under these conditions. We begin by showing that we have existence of a solution (which is a classical result), and moreover, that it is given as the limit of the $\mu_n$. The key estimate is the following.

\begin{Lemma}\label{Lem:WSquared}
For all $n$, we have that $\sum_{x\in M^{-n}\Gamma} \lp w_{n}(x) \rp^2= \frac{1}{2^{n}}$.
\end{Lemma}

\begin{proof}
We proceed by induction. The base case, for $w_1$, is simply equation \eqref{Eqn:OrthoConds} with $i=0$, by the definition of $\mu_1$. So we assume that the lemma is true for $w_{n-1}$.

We know that $\mu_{n-1}$ is supported on $M^{-(n-1)} \Gamma$ and that $\mu_n= (D_{\star})^{n-1}(\mu_1)\star \mu_{n-1}$. Thus, for any $x\in M^{-n}\Gamma$ (where we recall that $\supp(\mu_n)\subset M^{-n}\Gamma$), we have
\[
w_n(x) = \sum_{k\in\Gamma} p_k w_{n-1}\lp x-M^{-n}k\rp .
\]
where we extend $w_{n-1}$ to $M^{-n}\Gamma$ by letting it be zero off of $M^{-n+1}\Gamma$. 
Squaring this gives
\[
w_n(x)^2 = \sum_{k\in\Gamma}\sum_{\ell\in\Gamma} p_kp_{\ell} w_{n-1}\lp x-M^{-n}k\rp w_{n-1}\lp x-M^{-n}\ell\rp .
\]
We wish to reindex the sum in terms of the difference between $k$ and $\ell$. However, we also note that, because $w_{n-1}$ is supported on $M^{-(n-1)} \Gamma$, each term in this sum can only be non-zero if $k-\ell\in M\Gamma$. Thus we have
\[
w_n(x)^2 = \sum_{k\in\Gamma}\sum_{j\in\Gamma} p_kp_{k-Mj} w_{n-1}\lp x-M^{-n}k\rp w_{n-1}\lp x-M^{-n}k+M^{-n+1}j\rp .
\]
Summing over $x$, we note that there are only finitely many non-zero terms in any of these sums, and thus we are free to re-order them, so that
\[\begin{split}
\sum_{x\in M^{-n}\Gamma} w_n(x)^2 & = \sum_{k\in\Gamma}\sum_{j\in\Gamma} p_kp_{k-Mj} 
\sum_{x\in M^{-n}\Gamma} w_{n-1}\lp x-M^{-n}k\rp w_{n-1}\lp x-M^{-n}k+M^{-n+1}j\rp \\
& = \sum_{j\in\Gamma} \lp \lb \sum_{k\in\Gamma} p_kp_{k-Mj}\rb
\lb \sum_{x\in M^{-n}\Gamma} w_{n-1}\lp x\rp w_{n-1}\lp x+M^{-n+1}j\rp \rb \rp
\end{split}\]
where we've used the fact that the inner sum in the first line depends on $j$ but not on $k$. Finally, the orthonormality conditions \eqref{Eqn:OrthoConds} show that the sum over $k$ in the last line is $\frac{1}{2}\delta_{0j}$, so that, by the inductive hypothesis
\[
\sum_{x\in M^{-n}\Gamma} w_n(x)^2 = \frac{1}{2} \sum_{x\in M^{-n}\Gamma} w_{n-1}\lp x\rp^2 = \frac{1}{2}\cdot\frac{1}{2^{n-1}} .
\]
\end{proof}

\begin{Cor}\label{Cor:TVForOrtho} 
The total variation norms of the $\mu_n$ are uniformly bounded; in particular, if $N$ is the constant from equation \eqref{Eqn:CardOfSupport}, then $\|\mu_n\|_{\mathrm{TV}} \leq \sqrt{N}$ for all $n$. Thus $\mu_n\rightarrow \mu$.
\end{Cor}

\begin{proof}
Recall that $\mu_n$ is discrete. Then using the Cauchy-Schwarz inequality, equation \eqref{Eqn:CardOfSupport}, and Lemma \ref{Lem:WSquared}, we have
\[
\|\mu_n\|_{\mathrm{TV}} = \sum_{x\in S_n} |w_n(x)| 
 \leq \sqrt{\sum_{x\in M^{-n}\Gamma} \lp w_{n}(x) \rp^2} \sqrt{\card(S_n)} 
 \leq \sqrt{\frac{1}{2^{n}}} \sqrt{N \cdot 2^n} 
 = \sqrt{N} .
\]
\end{proof}

Under the orthogonality conditions, the solution measure $\mu$ always has a density, and further, this density is in $L^2$ (here, and in what follows, $L^q$ is defined relative to Lebesgue measure on either $\bR$ or $\bR^2$, as appropriate). Again, this is a classical result, but we give a proof in line with the techniques introduced above.

\begin{THM}\label{THM:L2Bound}
Under the orthogonality conditions, $\mu$ has a density $\phi$ (with respect to Lebesgue measure), and $\phi \in L^2$ with $||\phi||_2 \leq 1$.
\end{THM}

\begin{proof}
We begin by letting $f$ be a continuous smooth function with compact support. Then, by definition of $\mu$ and $\mu_n$, we have 
\begin{equation*}
\int f d \mu =\lim _{n \rightarrow \infty} \int f d \mu _n
 =\lim _{n \rightarrow \infty} \sum_{x \in S_n} f(x) w_n(x).
\end{equation*}

Now we can apply the Cauchy-Schwarz inequality to the sum on the right-hand side, followed by Lemma 4. So we have,
\begin{equation*}
\begin{split}
\int f d \mu  &\leq \lim _{n \rightarrow \infty} \sqrt{ \sum_{x \in S_n}w_n^2(x) \sum_{x \in S_n} f^2(x)}\\
\quad \quad &=\lim _{n \rightarrow \infty} \sqrt{ \frac{1}{2^{n}} \sum_{x \in S_n} f^2(x)} .
\end{split}
\end{equation*}
Since $2^{-n}\sum_{k\in\Gamma} \delta_{M^{-n}k}$ converges weakly to Lebesgue measure as $n\rightarrow\infty$, we have
\begin{equation*}
\int f d \mu \leq \sqrt{\int f^2(x) dx}
= \|f\|_2.
\end{equation*}
Therefore integration of a smooth function $f$ against $\mu$ is bounded by $\|f\|_2$.  By the density of such functions in $L^2$, we see that integration against $\mu$ determines a bounded linear functional on $L^2$. Since $L^2$ is a Hilbert space, we conclude that $\mu=\phi\, dx$ with $\phi \in L^2$ and $\|f\|_2\leq 1$.
\end{proof}

\begin{rem}
The total variation bound in Corollary \ref{Cor:TVForOrtho} implies the density $\phi$ is in $L^1$ with $\|\phi\|_1\leq \sqrt{N}$. However, the $L^2$ bound of Theorem \ref{THM:L2Bound} and the fact that  $\phi$ has compact support give, via H\"older's inequality, that
\[
\|\phi\|_1\leq \sqrt{\mathrm{Leb}(\supp(\phi))} ,
\]
where $\mathrm{Leb}$ denotes Lebesgue measure. Further, this is at least as good as the bound by $\sqrt{N}$, given the definition of $N$ and the fact, already noted, that $2^{-n}\sum_{k\in\Gamma} \delta_{M^{-n}k}$ converges weakly to Lebesgue measure.
\end{rem}

\begin{rem}
In light of the proofs of Lemmas \ref{Lem:WProb} and \ref{Lem:WSquared}, the conditions on the $p_k$ and the consequences for $\mu$ are not surprising. Indeed, the proof of Lemma \ref{Lem:WSquared} is based on an $L^2$-type bound on the $w_n$ which persists in the limit. Given how the $w_n$ act on smooth functions, $\mu$ should then correspond to an element of the dual space $(L^2)^*$, which, of course, is $L^2$. On the other hand, assuming that the $p_i$ give a probability measure is essentially an $L^1$ bound on the $w_n$, which means that $\mu$ should give an element of $(L^{\infty})^*$. And this is consistent with $\mu$ being a probability measure, but there is no reason to assume that $\mu$ will have an $L^1$-density, since $(L^{\infty})^*$ is strictly larger than $L^1$. However, the stronger assumption that the even and odd coefficients each sum to $1/2$ gives an $L^{\infty}$ bound on the $w_n$. Then we expect that $\mu$ will give an element of $(L^1)^*$, meaning that $\mu$ should have an $L^{\infty}$ density, which indeed is what we saw. While this might suggest considering other $L^p$-type bounds on the $w_n$, there do not appear to be any obvious, non-trivial conditions on the $p_k$ that lead to such bounds. 
\end{rem}

\section{Example of computing a scaling function on $\bR$}\label{Sect:OneDComp}

We consider the example of the scaling function corresponding to Daubechies' D4 wavelet. Here we work on $\bR$, and the dilation equation has pseudo-probabilities $p_0=\frac{1+\sqrt{3}}{8}$, $p_1=\frac{3+\sqrt{3}}{8}$, $p_2=\frac{3-\sqrt{3}}{8}$, and $p_3=\frac{1-\sqrt{3}}{8}$. From equation \eqref{Eqn:SuppOfMuN}, we easily see that $\supp(\mu) \subset [0, 3]$, so we apply the dilation equation to the intervals of length 1: $[0, 1]$, $[1, 2]$, and $[2, 3]$. By doing this, we obtain the following system of linear equations: 

\begin{align*}
\mu([0,1]) &=  p_0\mu([0,2]-0) + p_1\mu([0,2]-1) + p_2\mu([0,2]-2) + p_3\mu([0,2]-3)\\
	\quad		 &=  p_0\mu([0,1]) + p_0\mu([1,2]) + p_1\mu([-1,0]) + p_1\mu([0,1]) + p_2\mu([-2,-1])\\
	\quad 	 &\qquad + p_2\mu([-1,0]) + p_3\mu([-3,-2]) + p_3\mu([-2,-1])\\
	\quad		 &=  (p_0+p_1)\mu([0,1]) + p_0\mu([1,2]),\\
\mu([1, 2]) 		 &=  (p_2+p_3)\mu([0,1]) + (p_1+p_2)\mu([1,2]) + (p_0+p_1)\mu([2,3]), \quad\text{and}\\
\mu([2,3])  &=  p_3\mu([1,2]) + (p_2 + p_3)\mu([2, 3]).\\
\end{align*}

This system can be represented by the matrix

\[
 A= \left( \begin{array}{ccc}
p_0+p_1 & p_0		  & 0 \\
p_2+p_3 & p_1+p_2 & p_0+p_1 \\
0				& p_3		  & p_2+p_3 \end{array} \right)  
= \left( \begin{array}{ccc}
\frac{4+2\sqrt{3}}{8} & \frac{1+\sqrt{3}}{8}	  & 0 \\
\frac{4-2\sqrt{3}}{8} & \frac{3}{4} & \frac{4+2\sqrt{3}}{8}\\
0				& \frac{1-\sqrt{3}}{8}		  & \frac{4-2\sqrt{3}}{8} \end{array} \right).\\
\] 

We find that $A$ has eigenvalue $1$, and the corresponding right eigenspace is one-dimensional and is spanned by the vector
\[ V= \left(
\begin{array}{c}
\frac{-1-\sqrt{3}}{2-3\sqrt{3}}\\
\frac{2}{5+4\sqrt{3}}\\
\frac{\sqrt{3}-1}{-17-9\sqrt{3}}\\
\end{array}
\right).\\
\]
Here we have normalized $V$ so that its components sum to 1, which is equivalent to letting $\mu(\bR)=\mu([0,3])=1$. This tells us, specifically, that $\mu(0, 1)=\frac{-1-\sqrt{3}}{2-3\sqrt{3}}$, $\mu(1, 2)=\frac{2}{5+4\sqrt{3}}$, and $\mu(2, 3)=\frac{\sqrt{3}-1}{-17-9\sqrt{3}}$. In light of the fact that $\mu$ is absolutely continuous, this determines $\mu$ on $\Sigma_0$ (as already discussed); in more pseudo-probabilistic language, this determines the expectation of (the density of) $\mu$ conditioned on $\Sigma_0$. 

From here, we can recursively use the signed measure dilation equation to find the measures of the intervals of length $\frac{1}{2}$ (specifically $\lb0, \frac{1}{2}\rb$, $\lb \frac{1}{2}, 1 \rb$, $\lb 1, \frac{3}{2} \rb$, etc.), then $\frac{1}{4}$, and so on, which is equivalent to finding $\mu$ restricted to $\Sigma_1$, then $\Sigma_2$, and so on.  Applying the dilation equation to the intervals of length $1/2$ gives us (using numerical approximations for simplicity)
\begin{align*} 
\mu  \lb 0, \frac{1}{2} \rb  &= p_0\mu\lb0, 1\rb &= 0.290170901\\
\mu \lb \frac{1}{2}, 1 \rb &= p_0\mu[1, 2] + p_1\mu[0, 1]  &= 0.559508468\\
\mu \lb 1, \frac{3}{2} \rb &= p_0\mu[2, 3] + p_1\mu[1, 2] + p_2\mu[0, 1] &= 0.227670901\\
\mu \lb \frac{3}{2}, 2 \rb &= p_1\mu[2, 3] + p_2\mu[1, 2] + p_3\mu[0, 1] &= -0.061004234\\
\mu \lb 2, \frac{5}{2} \rb &= p_2\mu[2, 3] + p_3\mu[1, 2] &= -0.017841801\\
\mu \lb \frac{5}{2}, 3 \rb &= p_3\mu[2, 3] &= 0.001495766.
\end{align*}

Continuing to calculate $\mu$  for the natural intervals of length $\frac{1}{4}$ and $\frac{1}{8}$ yields the step function approximations of $\phi$ illustrated in Figure ~\ref{fig:Refining}.

\begin{figure}
\includegraphics[width=\linewidth]{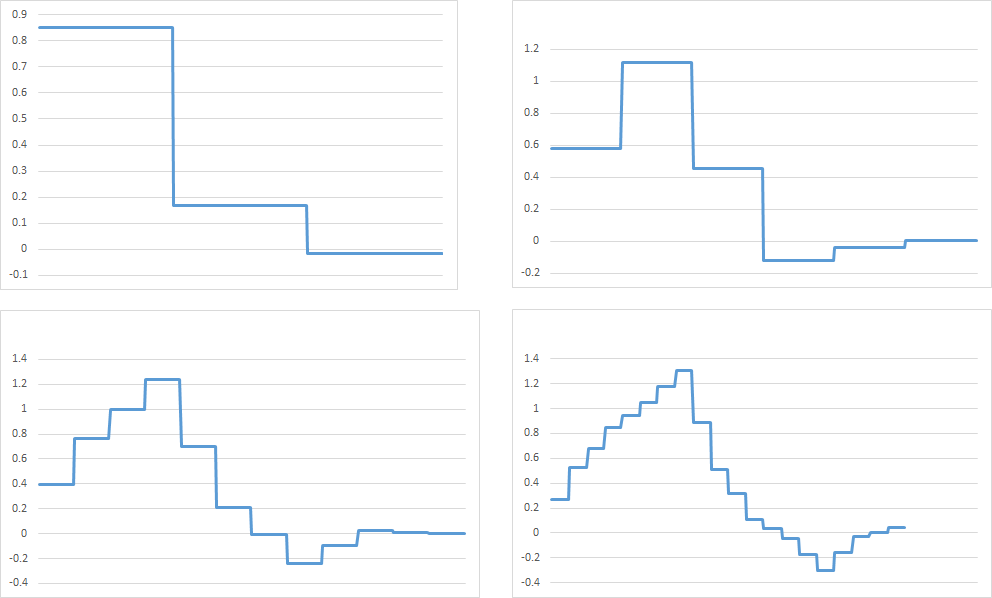}
\caption{Dyadic step-function approximations of scaling function D4. The value of the approximation on a dyadic interval is the measure $\mu$ of the interval divided by its length, that is, the density relative to $\Sigma_n$.}
\label{fig:Refining}
\end{figure}

\begin{rem}
Here we have used the standard tiles based on $[0,1]$, which is justified by the fact that we know $\mu$ will be absolutely continuous, since its density will be the D4 wavelet. However, to return to the considerations of the end of Section \ref{Sect:Prelim}, note that all of the computations remain valid if all intervals of the form $[a,b]$ are replaced with the intervals $(a,b]$, except that we then need to include the interval $(-1,0]$ in order to cover the entire support of $\mu$, namely $[0,3]$. However, we immediately see that the system of equations for $\mu(-1,0]$, $\mu(0,1]$, $\mu(1,2]$, and $\mu(2,3]$ can only be consistent if $\mu(-1,0]=0$. Since $\mu(-1,0]=\mu\{0\}$, this shows that $(-1,0]$ is $\mu$-null, and we are reduced to the above computations. Nonetheless, this illustrates concretely how the use of the modified tile $(0,1]$ allows one to consider solutions $\mu$ that may not have densities.
\end{rem}

\section{A correspondence of scaling functions}\label{Sect:GundyMethod}
As noted, it is helpful to compare the Twin Dragon tile, $T(+C_3, \mathcal{D}_1)$, with the unit interval $\{x: 0 \leq x \leq 1 \}$, $T(2, \mathcal{D}_1)$, viewed as a tile with dilation 2 and a digit set $\mathcal{D}_1=\{0, 1\}$. It is worth noting that the tiles $T(+C_3, \mathcal{D}_1)$ and $T(2, \mathcal{D}_1)$ both have the property that they contain exactly two lattice boundary points. The unit interval contains the points $0$ and $1$ while the twin dragon contains the points $0$ and $-i$. These similarities are significant for a procedure that maps a space of binary sequences into the spaces $\mathbb{Z}$ and $\mathbb{Z}^2$. This coding is generated by the pair $(2, \mathcal{D}_1)$ on one hand, and by $(+C_3, \mathcal{D}_1)$ on the other.

In one dimension, the coding procedure is performed simply by writing the real number in its binary representation. In two dimensions, the coding procedure is performed in a similar manner, except that the base for this representation is $1+i$. The first of these codings will map onto the non-negative half of the real line. Similarly, the latter will map onto half of the complex plane in some way (see \cite{Gundy10}). This common coding scheme is used in \cite{Gundy10} to give a correspondence between scaling functions in these two cases. In particular, this correspondence gives the principal result of \cite{Gundy10}, which states that if there exists a scaling function in one-dimension with coefficients $(p_k)$ in the dilation equation, then there exists a scaling function in two-dimensions with the same coefficients in the dilation equation with dilation by a factor of $M$, where $M$ belongs to the class $+C_3$.

It is worth noting, though, that this correspondence does not preserve continuity. While this is not surprising, in light of the construction, we take a moment to verify and underscore this point. Consider for example, the Daubechies' D4 scaling function from the last section. In one dimension, this family of scaling functions is notable for giving wavelets with regularity, and in particular, D4 is continuous. However, the corresponding scaling function in two dimensions fails to be continuous, as we now show. (Figure ~\ref{fig:D4Gundy} illustrates the transformation for the D4 scaling function on the interval $[0, 1]$ to the Twin Dragon.  The coloring in this figure represents the height of the function lying over the plane.)

\begin{figure}
  \includegraphics[width=\linewidth]{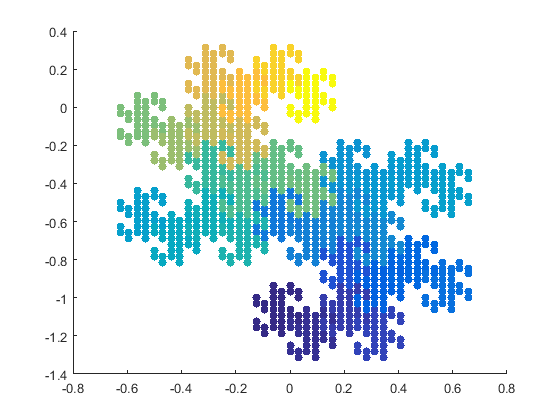}
  \caption{The D4 scaling function restricted to $[0, 1]$, approximated to the refinement on intervals of length $\frac{1}{2^5}$, translated to the plane. The coloring in this figure represents the height of the function lying over the plane; the darker blue representing higher values and the yellow representing lower values. The discontinuity, which can be seen at $(0, -0.5)$, persists in the limit. }
  \label{fig:D4Gundy}
\end{figure}

Let $\phi$ be the D4 scaling function (on $\bR$). First, we use the classical cascade algorithm to compute that 
\[
 \phi \lp \frac{1}{4} \rp  = \frac{2+ \sqrt{3}}{16} \quad\text{and}\quad
 \phi \lp \frac{3}{4} \rp   = \frac{3+2\sqrt{3}}{16}.
\]
The dyadic expansion of the points $\frac{1}{4}$ and $\frac{3}{4}$ are, respectively, $.01$ and $.10\bar{1}$. Both of these expansions correspond with the complex number $\frac{-i}{2}$ when using the base $1+i$. Let $T_{.01}$ denote the quarter-twin dragon which begins with the radix expansion $.01$ and $T_{.10}$ denote the quarter-twin dragon which begins with the radix expansion $.10$.  Then, if we consider the scaling function $\hat{\phi}$ on $\bR^2$ given by $\phi$ under the correspondence, we see that
\[\begin{split}
\lim_{x \rightarrow \frac{-i}{2} \text{ via } T_{.01}}\hat{\phi}(x) &=\frac{2+ \sqrt{3}}{16} \\
\text{while}\quad \lim_{x \rightarrow \frac{-i}{2} \text{ via } T_{.10}} \hat{\phi}(x) &=\frac{3+2\sqrt{3}}{16},
\end{split}\]
where the limits are understood as indicating that there is a sequence approaching $-i/2$ contained in the relevant tile along which this limiting behavior holds. In particular, $\hat{\phi}$ is discontinuous at $-i/2$.


\section{A cascade algorithm in two dimensions}\label{Sect:TwoDComp}

In this section, we restrict our attention to $\bR^2$. We recall that $T$ is the twin dragon, and that translates of $T$ by $\Zed^2$ tile the plane. Further, in this context we have that 
\[
 \Sigma_n=\sigma\lp\lc z+\frac{T}{(1+i)^n}\rc_{z\in \frac{1}{(1+i)^n}\Zed^2}\rp .
\]

We now consider two families of examples in $\bR^2$, one under the orthogonality assumptions of Section \ref{Sect:OrthoConds} and one under the probability assumptions of Section \ref{Sect:ProbConds}. In both cases, the goal is similar to that of Section \ref{Sect:OneDComp}, to illustrate a version of the cascade algorithm to compute the conditional expectations of $\mu$. The most interesting aspects are to determine which translates of $T$ intersect the support of $\mu$, the matrix $A$ relating the values of $\mu$ on these tiles, and the structure of the right 1-eigenspace of $A$. From here, the restriction of $\mu$ to the $\Sigma_n$ can be computed recursively just as in Section \ref{Sect:OneDComp}. However, determining the support of $\mu$ in two dimensions is more intricate than in one dimension. One reason is that the twin dragon $T$ is a more complicated shape than the unit interval $[0,1]$, but also, while $\sum_{k=0}^{2N-1} [0, 1] = [0, 2N]$, the analogue is not true in higher dimensions, $\sum_{k=0}^{2N-1} T \neq 2N \cdot T$.

In both families of examples, even though the support of $\{p_k\}$ contains no more than 4 points, the computations are sufficiently long that we use Sage to handle the (symbolic) linear algebra and a simple Python script to iteratively compute the points of $S_n$, as described in equation \eqref{Eqn:SuppOfMuN}.

\subsection{A four coefficient case under orthonormality conditions}

We consider that case when $\tilde{\Gamma}=\{0, 1, 1+i, 2+i\}$; that is when, $p_0$, $p_1$, $p_{1+i}$, and $p_{2+i}$ are the only non-zero coefficients. To begin, we make no other assumptions on the $p_j$, although we suppose that the $\mu_n$ converge weakly a (finite) signed measure $\mu$, absolutely continuous with respect to Lebesgue measure and satisfying the corresponding dilation equation, and we try to compute $\mu$ on $\Sigma_0$.

We begin by determining which $z\in \Zed^2$ are such that $z+T$ intersects $\supp(\mu_n)$ in a set of positive (Lebesgue) measure; denote the set of such $z$ by $\tilde{S}$. Such a $z$ must satisfy (see equation \eqref{Eqn:SuppOfMuN} and note that $\cup_{n=1}^{\infty}S_n$ is dense in the support of $\mu$)
\[
z+\sum_{n=1}^{\infty} \frac{\alpha_n}{(1+i)^n} = \sum_{n=1}^{\infty} \frac{\beta_n}{(1+i)^n}
\]
for some sequences of coefficients $\alpha_n\in\{0,1\}$ and $\beta_n\in\{0,1,1+i,2+i\}$. But that means there must exist a sequence $\tilde{\alpha}_n\in\{0,1,-1,2,i,1+i,2+i\}$ such that
\[
z = \sum_{n=1}^{\infty} \frac{\tilde{\alpha}_n}{(1+i)^n} .
\]
Then by the same basic estimate as used to compute $R$ in Section \ref{Sect:Support}, we see that no Gaussian integer with modulus greater than $\sqrt{5}\cdot \frac{\sqrt{2}}{\sqrt{2}-1} \approx 7.6344$ is in $\tilde{S}$. This leaves a finite set of candidates for $\tilde{S}$.

From the other direction, by computing $S_{12}$ and seeing which translated tile $z+T$ each of the points belongs to, we find 14 Gaussian integers that must be in $\tilde{S}$, which we give in order (as a row vector) as follows
\[
S = \lb 0, -i, 1-i, 1, 1+i, i, -1+i, -1, -1-i, -2i, 1-2i, 2-2i, 2-i, 2 \rb .
\]

Next, we consider the Gaussian integers with modulus less than 8 but which are not included in $S$. The idea is as follows. Suppose we pick a Gaussian integer $z$ and apply the dilation equation some number of times. This will express $\mu\lp z+T\rp$ as a linear combination of the measures of some other translated tiles, say $\mu\lp z_n+T\rp$ for $1\leq n\leq N$. But if $|z_n|\geq 8$ for all $n$, then by our previous remarks about $\supp(\mu)$, we have $\mu\lp z_n+T\rp=0$ for $1\leq n\leq N$. Then $\mu\lp z+T\rp=0$, and further, this reasoning applies to any Borel subset of $z+T$. Thus $z\not\in\tilde{S}$. We say that such a Gaussian integer $z$ is ``pushed out.'' If we carry out this procedure for every Gaussian integer with modulus less than or equal to 8 that is not in $\set(S)$ (where $\set(S)$ denotes the components of the vector $S$, consider as a set), we find there are 14 such Gaussian integers which don't get pushed out. We give these 14 points, in order, as
\begin{align*}
S^{\prime} = [ &2i, -1+2i, -2+i, -2, -2-i, -1-2i, -3i,\\ 
		\quad & 1-3i, 2-3i, 3-2i, 3-i,  3, 2+i, 1+2i ] .\\
\end{align*}

Now we know that $\tilde{S}\subset\set(S)\cup\set(S')$. Consider the 28 points given in order by $S\oplus S^{\prime}$, which is the vector whose first 14 components are given by $S$ and whose last 14 components are given by $S^{\prime}$. We use $(S)_n$ to denote the $n$th component of the vector $S$. Let $V = v^{S}\oplus v^{S^{\prime}}$ be the vector of real numbers, the $k$th entry of which is $\mu(\lp S\oplus S^{\prime} \rp _k+T)$ for $1\leq k\leq 28$. We are interested in computing the entries of $V$. 

\begin{rem}Even for the four points $0$, $1$, $1+i$, and $2+i$, the size of the vectors and matrices under consideration is unwieldy. So even though $V$ is more natural as a column vector, we write it as a row vector to save space and then transpose it as necessary. 
\end{rem}

First note that $(1+i)T = T\cup (1+T)$. Thus, from the dilation equation, we have
\[\begin{split}
\mu\lp T\rp & =  p_0 \mu\lp T\cup(1+T) \rp +  p_1 \mu\lp (-1+T)\cup(T) \rp +\\
& \quad\quad p_{1+i}  \mu\lp (-i-1+T)\cup(-i+T)\rp + p_{2+i} \mu\lp (-i-2+T)\cup(-i-1+T)  \rp .
\end{split}\]
Because any two translates of $T$ by distinct Gaussian integers are disjoint up to a set of Lebesgue measure zero (by the tiling property), and $\mu$ is absolutely continuous with respect to Lebesgue measure, we can split each of the four terms of the right-hand side of the above equation to get
\[\begin{split}
\mu\lp T\rp & =  p_0 \mu \lp T \rp + p_0 \mu \lp 1+T \rp +  p_1 \mu \lp -1+T \rp + p_1 \mu \lp T \rp \\
& \quad\quad + p_2  \mu \lp -i-1+T \rp + \mu \lp -i+T \rp + p_3 \mu \lp -i-2+T \rp + p_3 \mu \lp -i-1+T  \rp .
\end{split}\]
Further, all of the translates on the right-hand side belong to $\set(S) \cup \set(S^{\prime})$, so none of these are zero a priori, so we have (in terms of $V$):
\[
V_1 = \lp p_0+p_1\rp V_1 + p_{1+i} V_2 + p_0 V_4 + p_1 V_8 + \lp p_{1+i} +p_{2+i} \rp V_9 + p_{2+i} V_{19} .
\] 
A similar computation can be performed for the other 27 components of $V$, where any translated tile involving a shift by a Gaussian integer outside of $\set(S)\cup \set(S^{\prime})$ that appears is discarded. The result is a system of 28 linear equations, which we can write as
\[
\lp V\rp^{\intercal} = \hat{A}(p_0,p_1,p_{1+i},p_{2+i}) \lp V\rp^{\intercal},
\]
where $ \hat{A}(p_0,p_1,p_{1+i},p_{2+i})$ is a $28\times 28$ matrix. We find that $\hat{A}$ has a block upper-triangular decomposition as
\[
\hat{A} =
\begin{bmatrix}
A & * \\
0 & A^{\prime}
\end{bmatrix},
\]
where, if we let $P_{0,1}=p_0+p_1$ and $P_{2,3}=p_{1+i}+p_{2+i}$, $A$ is
\[
\begin{bmatrix}
P_{0,1} & p_{1+i} & 0 & p_{0} & 0 & 0 & 0 & p_{1} & P_{2,3} & 0 & 0 & 0 & 0 & 0 \\
0 & p_{1} & P_{0,1} & 0 & 0 & 0 & 0 & 0 & 0 & P_{2,3} & p_{1+i} & 0 & p_{0} & 0 \\
0 & p_{2+i} & P_{2,3} & p_{1} & 0 & 0 & 0 & 0 & 0 & 0 & 0 & 0 & p_{1+i} & P_{0,1} \\
P_{2,3} & 0 & 0 & p_{1+i} & P_{0,1} & p_{1} & 0 & p_{2+i} & 0 & 0 & 0 & 0 & 0 & 0 \\
0 & 0 & 0 & 0 & 0 & p_{1+i} & P_{2,3} & 0 & 0 & 0 & 0 & 0 & 0 & 0 \\
0 & 0 & 0 & 0 & 0 & p_{0} & P_{0,1} & p_{1+i} & 0 & 0 & 0 & 0 & 0 & 0 \\
0 & 0 & 0 & 0 & 0 & 0 & 0 & p_{0} & 0 & 0 & 0 & 0 & 0 & 0 \\
0 & p_{0} & 0 & 0 & 0 & 0 & 0 & 0 & P_{0,1} & 0 & 0 & 0 & 0 & 0 \\
0 & 0 & 0 & 0 & 0 & 0 & 0 & 0 & 0 & P_{0,1} & p_{0} & 0 & 0 & 0 \\
0 & 0 & 0 & 0 & 0 & 0 & 0 & 0 & 0 & 0 & p_{1} & P_{0,1} & 0 & 0 \\
0 & 0 & 0 & 0 & 0 & 0 & 0 & 0 & 0 & 0 & p_{2+i} & P_{2,3} & p_{1} & 0 \\
0 & 0 & 0 & 0 & 0 & 0 & 0 & 0 & 0 & 0 & 0 & 0 & p_{3} & 0 \\
0 & 0 & 0 & p_{2+i} & 0 & 0 & 0 & 0 & 0 & 0 & 0 & 0 & 0 & P_{2,3} \\
0 & 0 & 0 & 0 & P_{2,3} & p_{2+i} & 0 & 0 & 0 & 0 & 0 & 0 & 0 & 0
\end{bmatrix} ,
\]
$A^{\prime}$ is
\[
\begin{bmatrix}
0 & p_{0} & p_{1+i} & 0 & 0 & 0 & 0 & 0 & 0 & 0 & 0 & 0 & 0 & 0 \\
0 & 0 & p_{0} & 0 & 0 & 0 & 0 & 0 & 0 & 0 & 0 & 0 & 0 & 0 \\
0 & 0 & 0 & 0 & p_{0} & 0 & 0 & 0 & 0 & 0 & 0 & 0 & 0 & 0 \\
0 & 0 & 0 & 0 & 0 & p_{0} & 0 & 0 & 0 & 0 & 0 & 0 & 0 & 0 \\
0 & 0 & 0 & 0 & 0 & 0 & p_{0} & 0 & 0 & 0 & 0 & 0 & 0 & 0 \\
0 & 0 & 0 & 0 & 0 & 0 & p_{1} & P_{0, 1} & p_{0} & 0 & 0 & 0 & 0 & 0 \\
0 & 0 & 0 & 0 & 0 & 0 & 0 & 0 & p_{1} & 0 & 0 & 0 & 0 & 0 \\
0 & 0 & 0 & 0 & 0 & 0 & 0 & 0 & p_{2+i} & p_{1} & 0 & 0 & 0 & 0 \\
0 & 0 & 0 & 0 & 0 & 0 & 0 & 0 & 0 & p_{2+i} & 0 & 0 & 0 & 0 \\
0 & 0 & 0 & 0 & 0 & 0 & 0 & 0 & 0 & 0 & 0 & p_{2+i} & 0 & 0 \\
0 & 0 & 0 & 0 & 0 & 0 & 0 & 0 & 0 & 0 & 0 & 0 & p_{2+i} & 0 \\
0 & 0 & 0 & 0 & 0 & 0 & 0 & 0 & 0 & 0 & 0 & 0 & 0 & p_{2+i} \\
P_{2, 3} & p_{2+i} & 0 & 0 & 0 & 0 & 0 & 0 & 0 & 0 & 0 & 0 & 0 & p_{1+i} \\
0 & p_{1+i} & 0 & 0 & 0 & 0 & 0 & 0 & 0 & 0 & 0 & 0 & 0 & 0
\end{bmatrix},
\]
$*$ is a $14\times14$ block that we don't compute explicitly (for reasons that will be clear in a moment), and ``$0$'' is the $14\times 14$ zero matrix.

\begin{rem}The order of the components in $S$ is chosen so that they spiral out clockwise as points in the plane. This is related to the action multiplication by $1+i$ on the plane, and, maybe more importantly, gives $A$ and the structure seen above of one roughly diagonal band and one roughly above-diagonal band.
\end{rem}

Next, we compute that 
\[
\det\lp A^{\prime}-I\rp = 1-p_0^3 p_1 p_{1+i} p_{2+i}^3
\]
(where $I$ is the $14\times 14$ identity matrix). Note that under either the orthogonality conditions of Section \ref{Sect:OrthoConds} or the probability conditions of Section \ref{Sect:ProbConds}, the $p_k$ all have absolute value no more than 1, and at most one of them can have absolute value 1. Thus, under either set of conditions, this determinant is always positive, and in particular 1 cannot be an eigenvalue of $A^{\prime}$. This means that in order for $\lp V\rp^{\intercal}$ to be a right 1-eigenvector for $\hat{A}(p_0,p_1,p_{2},p_{3})$, we must have that $v^{S^{\prime}}$ is the zero vector. One consequence of this (and the upper triangular block structure of $\hat{A}$) is that $\set(S^{\prime})$ is not in $\tilde{S}$. In other words, we have that $\tilde{S}=\set(S)$. (Note that coincides nicely with what we saw by considering points which had ``binary'' expansions up to 12 places.) Further, we see that we only need to consider the system
\begin{equation}\label{Eqn:For4Coefs}
\lp v^S\rp^{\intercal} = A(p_0,p_1,p_{1+i},p_{2+i}) \lp v^S\rp^{\intercal} ,
\end{equation}
for the $14\times 14$ matrix $A$ given above. (This is the reason we don't bother to compute the upper right block ``*'', and the reason we chose our notation in this way.) The right 1-eigenspace is always at least one-dimensional, since the vector of all 1's is a left 1-eigenvector. We can see this easily for the $A$ above, since each column sums to $p_0+p_1+p_{1+i}+p_{2+i}=1$.

For concreteness, let us now assume that $p_0$, $p_1$, $p_{1+i}$, and $p_{2+i}$ satisfy the orthogonality conditions of Section \ref{Sect:OrthoConds}. Then all of our assumptions on $\mu$ are satisfied as well, and we see that $\mu$ restricted to $\Sigma_0$ is given by a vector $v^S$ satisfying equation \eqref{Eqn:For4Coefs}. Thus, if the right 1-eigenspace of $A$ is 1-dimensional, it determines $\mu$, up to scaling, in the sense that, given $\mu(\bR^2)$, $\mu$ restricted to $\Sigma_0$ is given by the unique corresponding basis for the right 1-eigenspace, and then $\mu$ is determined recursively as discussed above.

To illustrate this concretely, we now specialize to the case when $p_0=(1+\sqrt{3})/8$, $p_1= (3+\sqrt{3})/8$, $p_2=(3-\sqrt{3})/8$, and $p_3=(1-\sqrt{3})/8$ (in order to to mimic the one-dimensional D4 case). We find that $A$ has a 1-dimensional right 1-eigenspace, which is spanned by
\[\begin{split}
\big[ & 0.988473215486,\,0.0991927845318,\,0.0476287661136,\,0.0956000986321,\, \\
& 0.00421010706117,\,0.0221507197936,\,0.0104401680866,\, \\
& 0.0305709339159,\,-0.00354125079226,\,-0.00205532069064,\,-0.00475426145644,\, \\
& 0.000811194910171,\,-0.00886490283771,\,-0.00174490784237 \big] .
\end{split}\]
(Note that we have switched to a decimal approximation in order for this vector to fit on the page.) So this gives $\lp v^S\rp^{\intercal} $, and hence $\mu|_{\Sigma_0}$, uniquely up to scaling.

\begin{rem}
As in Section \ref{Sect:OneDComp}, these computations could be extended in a straightforward way to the case when $\mu$ is not assumed to be absolutely continuous by replacing $T$ with a suitable modification (analogous to replacing $[0,1]$ with $(0,1]$). However, as indicated at the end of Section \ref{Sect:Prelim}, we are unaware of any work giving such a modification of the twin dragon.
\end{rem}

\subsection{A three coefficient case under the probability conditions}

We summarize the results in a parallel way to the previous section, and in particular, we use the parallel notation in what follows. Now we assume that $\tilde{\Gamma}=\{0, 1, 1+i, 2+i\}$; that is, that $p_0$, $p_1$, and $p_{i}$ are the only non-zero coefficients.
We also again assume that $\mu$ is a (finite) signed measure on $\sB\lp \bR^2\rp$ such that $\mu$ is absolutely continuous with respect to Lebesgue measure and satisfies the dilation equation associated to $p_0$, $p_1$, and $p_i$.

We begin by determining $\tilde{S}$. First, no Gaussian integer with modulus greater than $\frac{2}{\sqrt{2}-1} \approx 4.828$ is in $\tilde{S}$. Further, by computing
\[
\lc \sum_{n=1}^{12} \frac{\tilde{\gamma}_n}{\lp 1+i\rp^n} : \tilde{\gamma}_n\in\lc 0,1,i \rc \rc ,
\]
we see that there are at least 16 Gaussian integers in $\tilde{S}$. Namely, $\tilde{S}$ contains $\set(S)$ where $S$ is the vector
\[
S=  \lb 0 , i , 1+i , 1 , 1-i , -i , -1 , -1+i , 2i , 1+2i , 2+i , 2 , 2-i , 1-2i , -2i , -1-i \rb .
\]
If we now look at all the Gaussian integers with modulus no greater than $\frac{2}{\sqrt{2}-1}$ which are not in $\set(S)$, we see that (after applying the dilation equation, say 10 times) they all ``get pushed'' outside of the ball of radius 5 centered at the origin. In other words, we see that $\tilde{S}=\set(S)$ (and there is no analogue of $S'$ in this situation).

We again consider the vector $v^S$. The dilation equation gives
\[
\lp v^S\rp^{\intercal} = A(p_0,p_1,p_{i}) \lp v^S\rp^{\intercal} ,
\]
where $A(p_0,p_1,p_i)$ is the $16\times16$ matrix given by
\[
\begin{bmatrix}
P_{0,1} & 0 & 0 & p_{0} & p_i & p_i & p_{1} & 0 & 0 & 0 & 0 & 0 & 0 & 0 & 0 & 0 \\
p_i & p_{0} & 0 & 0 & 0 & 0 & p_i & P_{0,1} & 0 & 0 & 0 & 0 & 0 & 0 & 0 & 0 \\
0 & p_i & p_i & 0 & 0 & 0 & 0 & 0 & p_{0,1} & p_{0} & 0 & 0 & 0 & 0 & 0 & 0 \\
0 & p_{1} & P_{0,1} & p_i & 0 & 0 & 0 & 0 & 0 & 0 & p_{0} & p_i & 0 & 0 & 0 & 0 \\
0 & 0 & 0 & p_{1} & 0 & 0 & 0 & 0 & 0 & 0 & 0 & P_{0,1} & p_i & 0 & 0 & 0 \\
0 & 0 & 0 & 0 & P_{0,1} & p_{1} & 0 & 0 & 0 & 0 & 0 & 0 & p_{0} & p_i & 0 & 0 \\
0 & 0 & 0 & 0 & 0 & p_{0} & 0 & 0 & 0 & 0 & 0 & 0 & 0 & 0 & p_i & P_{0,1} \\
0 & 0 & 0 & 0 & 0 & 0 & p_{0} & 0 & 0 & 0 & 0 & 0 & 0 & 0 & 0 & p_i \\
0 & 0 & 0 & 0 & 0 & 0 & 0 & p_i & 0 & 0 & 0 & 0 & 0 & 0 & 0 & 0 \\
0 & 0 & 0 & 0 & 0 & 0 & 0 & 0 & p_i & 0 & 0 & 0 & 0 & 0 & 0 & 0 \\
0 & 0 & 0 & 0 & 0 & 0 & 0 & 0 & 0 & p_i & 0 & 0 & 0 & 0 & 0 & 0 \\
0 & 0 & 0 & 0 & 0 & 0 & 0 & 0 & 0 & p_{1} & p_i & 0 & 0 & 0 & 0 & 0 \\
0 & 0 & 0 & 0 & 0 & 0 & 0 & 0 & 0 & 0 & p_{1} & 0 & 0 & 0 & 0 & 0 \\
0 & 0 & 0 & 0 & 0 & 0 & 0 & 0 & 0 & 0 & 0 & 0 & p_{1} & 0 & 0 & 0 \\
0 & 0 & 0 & 0 & 0 & 0 & 0 & 0 & 0 & 0 & 0 & 0 & 0 & p_{1} & 0 & 0 \\
0 & 0 & 0 & 0 & 0 & 0 & 0 & 0 & 0 & 0 & 0 & 0 & 0 & p_{0} & P_{0,1} & 0
\end{bmatrix}
\]
and $P_{0, 1}=p_0 + p_1$.

So $v^S$ must be a right 1-eigenvector of $A$, and if the right 1-eigenspace is 1-dimensional, then $v^S$, and hence $\mu$, must be unique up to scaling.

To complement the previous section, we now assume that $p_0$, $p_1$, and $p_i$ satisfy the probability conditions of Section \ref{Sect:ProbConds}. Then we know that $\mu$ will be absolutely continuous if $p_0=\frac{1}{2}$ and $p_1$ and $p_i$ are strictly between 0 and 1 and sum to $1/2$. Restricting our attention to this case, we have one degree of freedom; namely, set $p_1=(1/2)-p_i$ and let $p_i\in(0,1/2)$.  Making these substitutions explicitly in $A$ makes the resulting matrix too big to include here. However, for any value of $p_i$, the resulting matrix has a 1-dimensional right 1-eigenspace.

Specializing even further, we take the concrete example with $p_0=1/2$ and $p_1=p_i=1/4$. Then we get that $A(1/2,1/4,1/4)$ is
\[
\begin{bmatrix}
\frac{3}{4} & 0 & 0 & \frac{1}{2} & \frac{1}{4} & \frac{1}{4} & \frac{1}{4} & 0 & 0 & 0 & 0 & 0 & 0 & 0 & 0 & 0 \\
\frac{1}{4} & \frac{1}{2} & 0 & 0 & 0 & 0 & \frac{1}{4} & \frac{3}{4} & 0 & 0 & 0 & 0 & 0 & 0 & 0 & 0 \\
0 & \frac{1}{4} & \frac{1}{4} & 0 & 0 & 0 & 0 & 0 & \frac{3}{4} & \frac{1}{2} & 0 & 0 & 0 & 0 & 0 & 0 \\
0 & \frac{1}{4} & \frac{3}{4} & \frac{1}{4} & 0 & 0 & 0 & 0 & 0 & 0 & \frac{1}{2} & \frac{1}{4} & 0 & 0 & 0 & 0 \\
0 & 0 & 0 & \frac{1}{4} & 0 & 0 & 0 & 0 & 0 & 0 & 0 & \frac{3}{4} & \frac{1}{4} & 0 & 0 & 0 \\
0 & 0 & 0 & 0 & \frac{3}{4} & \frac{1}{4} & 0 & 0 & 0 & 0 & 0 & 0 & \frac{1}{2} & \frac{1}{4} & 0 & 0 \\
0 & 0 & 0 & 0 & 0 & \frac{1}{2} & 0 & 0 & 0 & 0 & 0 & 0 & 0 & 0 & \frac{1}{4} & \frac{3}{4} \\
0 & 0 & 0 & 0 & 0 & 0 & \frac{1}{2} & 0 & 0 & 0 & 0 & 0 & 0 & 0 & 0 & \frac{1}{4} \\
0 & 0 & 0 & 0 & 0 & 0 & 0 & \frac{1}{4} & 0 & 0 & 0 & 0 & 0 & 0 & 0 & 0 \\
0 & 0 & 0 & 0 & 0 & 0 & 0 & 0 & \frac{1}{4} & 0 & 0 & 0 & 0 & 0 & 0 & 0 \\
0 & 0 & 0 & 0 & 0 & 0 & 0 & 0 & 0 & \frac{1}{4} & 0 & 0 & 0 & 0 & 0 & 0 \\
0 & 0 & 0 & 0 & 0 & 0 & 0 & 0 & 0 & \frac{1}{4} & \frac{1}{4} & 0 & 0 & 0 & 0 & 0 \\
0 & 0 & 0 & 0 & 0 & 0 & 0 & 0 & 0 & 0 & \frac{1}{4} & 0 & 0 & 0 & 0 & 0 \\
0 & 0 & 0 & 0 & 0 & 0 & 0 & 0 & 0 & 0 & 0 & 0 & \frac{1}{4} & 0 & 0 & 0 \\
0 & 0 & 0 & 0 & 0 & 0 & 0 & 0 & 0 & 0 & 0 & 0 & 0 & \frac{1}{4} & 0 & 0 \\
0 & 0 & 0 & 0 & 0 & 0 & 0 & 0 & 0 & 0 & 0 & 0 & 0 & \frac{1}{2} & \frac{3}{4} & 0
\end{bmatrix}.
\]
The right 1-eigenspace is one-dimensional and is spanned by the vector
\[\begin{split}
& \lb 1  ,\,\frac{767262}{1370695},\,\frac{795934}{4112085},\,\frac{104324}{274139},\, \frac{130917}{1370695},\,\frac{131013}{1370695},\,\frac{13105}{274139},\,\frac{32768}{1370695}, \right. \\
&\qquad \left. \frac{8192}{1370695},\,\frac{2048}{1370695},\,\frac{512}{1370695},\,\frac{128}{274139},\, \frac{128}{1370695},\,\frac{32}{1370695},\,\frac{8}{1370695},\,\frac{22}{1370695} \rb .
\end{split}\]
This is normalized so that the first coordinate is 1, rather than being normalized to give a probability. Nonetheless, this determines $v^S$ and $\mu|_{\Sigma_0}$ up to normalization.

\bibliographystyle{amsplain}

\providecommand{\bysame}{\leavevmode\hbox to3em{\hrulefill}\thinspace}
\providecommand{\MR}{\relax\ifhmode\unskip\space\fi MR }
\providecommand{\MRhref}[2]{%
  \href{http://www.ams.org/mathscinet-getitem?mr=#1}{#2}
}
\providecommand{\href}[2]{#2}

\end{document}